\documentclass[a4paper,10pt]{amsart}
\usepackage{amsmath,amssymb,amsfonts,amsthm}
\usepackage{latexsym,mathrsfs}
\usepackage[all]{xy}
\usepackage{longtable}
\usepackage{supertabular}
\input xypic
\theoremstyle{plain}
\newtheorem{theorem}[subsection]{{\bf Theorem}}
\newtheorem*{theorem*}{{\bf Theorem}}
\newtheorem{corollary}[subsection]{{\bf Corollary}}
\newtheorem*{corollary*}{{\bf Corollary}}
\newtheorem{proposition}[subsection]{{\bf Proposition}}
\newtheorem{lemma}[subsection]{{\bf Lemma}}

\theoremstyle{definition}

\theoremstyle{remark}

\numberwithin{equation}{subsection}

\DeclareMathOperator{\HH}{H}

\DeclareMathOperator{\B}{B}

\DeclareMathOperator{\M}{M}

\DeclareMathOperator{\res}{res}

\newcommand{\QZ}{\mathbb{Q}/\mathbb{Z}}

\begin{document}
\title[Unramified Brauer groups]
{Groups of order $p^5$ and their unramified Brauer groups}
\author{Primo\v z Moravec}
\address{{
Department of Mathematics \\
University of Ljubljana \\
Jadranska 21 \\
1000 Ljubljana \\
Slovenia}}
\email{primoz.moravec@fmf.uni-lj.si}
\thanks{}
\date{\today}
\begin{abstract}
\noindent
We classify all groups of order $p^5$ with non-trivial unramified Brauer groups.
We show that if $p>3$, then there are precisely $\gcd (p-1,4)+\gcd (p-1,3)+1$
such groups.
\end{abstract}
\maketitle
\section{Introduction}
\label{s:intro}

\noindent
Let $G$ be a finite group and $V$ a faithful representation of $G$ over $\mathbb{C}$. Then there is a natural action of $G$ upon
the field of rational functions $\mathbb{C}(V)$. A problem posed by Emmy Noether \cite{Noe16} asks as to whether 
the field of $G$-invariant functions $\mathbb{C}(V)^G$ is purely transcendental over $\mathbb{C}$, i.e.,
whether the quotient space $V/G$ is {\it rational}. A question related to the above mentioned
is whether $V/G$ is {\it stably rational}, that is,
whether there exist independent variables $x_1,\ldots ,x_r$ such that $\mathbb{C}(V)^G(x_1,\ldots ,x_r)$ becomes a pure
transcendental extension of $\mathbb{C}$. This problem has close connection
with L\"{u}roth's problem \cite{Saf91} and the inverse Galois problem \cite{Swa83,Sal84}.
By Hilbert's Theorem 90 stable rationality of $V/G$
does not depend upon the choice of $V$, but only on the group $G$. Saltman \cite{Sal84} found examples of groups $G$ of order $p^9$ such
that $V/G$ is not stably rational over $\mathbb{C}$. His main method was application of the unramified cohomology group
$\HH^2_{\rm nr}(\mathbb{C}(V)^G,\QZ )$ as an obstruction. 
A version of this invariant had been used before by Artin and Mumford \cite{Art72} who constructed unirational varieties over
$\mathbb{C}$ that were not rational.
Bogomolov \cite{Bog88} further explored this cohomology group. He proved 
that $\HH^2_{\rm nr}(\mathbb{C}(V)^G,\QZ )$ is canonically isomorphic to 
\begin{equation}
\label{eq:b0def}
\B_0(G)=\bigcap _{\tiny\begin{matrix}A\le G,\\ A \hbox{ abelian}\end{matrix}} \ker \res ^G_A,
\end{equation}
where $\res ^G_A:\HH^2(G,\QZ )\to \HH^2(A,\QZ )$ is the usual cohomological restriction map.
The group $\B_0(G)$ is a subgroup of the {\it Schur multiplier} $\HH ^2(G,\QZ )$ of $G$. Kunyavski\u\i ~\cite{Kun08} coined the term the
{\it Bogomolov multiplier} of $G$ for the group $\B_0(G)$.
Bogomolov used the above description to find
new examples of groups $G$ of order $p^6$ with $\HH^2_{\rm nr}(\mathbb{C}(V)^G,\QZ )\neq 0$. Subsequently,
Bogomolov, Maciel and Petrov \cite{Bog04} showed that $\B_0(G)=0$ when $G$ is a finite simple group of Lie type $A_\ell$, whereas
Kunyavski\u\i \ \cite{Kun08} recently proved that $\B_0(G)=0$ for every quasisimple or almost simple group $G$. 

The Bogomolov multiplier of a given finite group is hard to compute. To illustrate this, we note that Chu, Hu, Kang, and Kunyavski\u\i \ \cite{Chu09}
only recently computed Bogomolov multipliers of all groups of order 64. In our paper \cite{Mor11}, we have obtained a new description of $\B _0(G)$
for a finite group $G$. It relies on the notion of the {\it nonabelian exterior square}
$G\wedge G$ of the group $G$ (see Section \ref{s:exterior} for definition and further details). 
To state the main result of \cite{Mor11}, denote by
$\M (G)$ be the kernel of the commutator
homomorphism $\kappa :G\wedge G\to[G,G]$, and $\M _0(G)=\langle x\wedge y\mid x,y\in G, [x,y]=1\rangle$. Then $\B _0(G)$ is isomorphic
to $\M (G)/\M_0(G)$. One of the advantages of this approach is that it provides a purely combinatorial description of $\B _0(G)$ in terms of 
presentation of $G$. Among other things, this leads to a Hopf-type formula for $\B _0(G)$, and an efficient algorithm for computing $\B_0(G)$ when
$G$ is a finite solvable group. This method enables fast computer calculations of $\B_0(G)$ for a reasonably large finite solvable group $G$, cf.
\cite{Mor11} for further details.

In his seminal paper \cite{Bog88}, Bogomolov claimed that if $G$ is a group of order $p^5$, where $p$ is a prime, then $\B_0(G)$ is trivial. 
Moreover, it had been conjectured that if $G$ is a group of order $p^5$, then $\mathbb{C}(V)^G$ is always purely transcendental over $\mathbb{C}$.
This is indeed true for
$p=2$ and was confirmed by Chu, Hu, Kang, and Prokhorov \cite{Chu08}. On the other hand, we showed \cite{Mor11} that there are precisely
three groups of order $3^5$ with nontrivial Bogomolov multiplier. This shows that both Bogomolov's claim and the above conjecture are false. 
More recently, Hoshi and Kang \cite{Hos11} proved that for every odd prime
$p$ there exist groups $G$ of order $p^5$ with $\B _0(G)\neq 0$. They posed a problem of classification of all such groups $G$. 

The purpose of this paper is to classify all groups of order $p^5$ with nontrivial Bogomolov multiplier. A complete list of all groups of order
$p^5$ is known by the work of James \cite{Jam80}. In his paper, the nonabelian groups of order $p^5$ are divided into isoclinism
families $\Phi _k$, $2\le k\le 10$. Hoshi and Kang \cite{Hos11} proved the following result.

\begin{theorem}[Hoshi, Kang, \cite{Hos11}]
\label{t:hoshi}
Let $G$ be a group of order $p^5$ that belongs to the family $\Phi _{10}$. Then $\B_0(G)\neq 0$.
\end{theorem}

Our main result shows that these groups are essentially the only ones having non-trivial Bogomolov multiplier:

\begin{theorem}
\label{t:p5b0}
Let $G$ be a group of order $p^5$, $p>3$. Then $\B _0(G)\neq 0$ if and only if $G$ belongs to the family $\Phi _{10}$. There are
precisely $\gcd (p-1,4)+\gcd (p-1,3)+1$ non-isomorphic groups with this property.
\end{theorem}

The case $p=3$ is already dealt with in \cite{Hos11} and \cite{Mor11}, where it is also proved that a group order $3^5$ has non-trivial Bogomolov multiplier 
if and only if it belongs to $\Phi _{10}$. For the sake of simplicity of the proofs we therefore assume that $p>3$.
The main step in the proof of the above theorem is Proposition \ref{p:class3} which ensures a sufficient condition for a $p$-group of nilpotency class $\le 3$ to have trivial Bogomolov multiplier. This condition can be read off from a polycyclic presentation of the group. The result we obtain covers the families
$\Phi _2$, $\Phi _3$, $\Phi _4$, and $\Phi _6$, which form the major part of groups in question. 
The remaining families are then dealt with separately. The techniques we use are purely combinatorial, and do not require any cohomological machinery.
The fact that every group in $\Phi _{10}$ has nontrivial unramified Brauer group follows from  Hoshi and Kang \cite{Hos11}. 

In 1987, Bogomolov \cite{Bog88} announced a full classificiation of groups of order $p^6$ with nontrivial $\B_0$, yet no such classification has appeared.
We note here that all groups of order $p^6$ have been listed by James \cite{Jam80}. Our techniques could be applied and extended to find all groups among these that
have nontrivial unramified Brauer group, but the calculations would become quite lengthy.

We have recently become aware of a paper by Hoshi, Kang, and Kunyavski\u\i \ \cite{Hos12} where another classification of groups of order $p^5$ with non-trivial Bogomolov multiplier has appeared. In fact, their results are more general and also deal with isomorphisms of the corresponding fields.
We note here that our techniques are essentially different, and our argument is much shorter.

\subsection*{Notations}
Let $G$ be a group and $x,y\in G$.  We use the notation $x^y=y^{-1}xy$ for conjugation from the right. 
The commutator $[x,y]$ of elements $x$ and $y$ is defined by
$[x,y]=x^{-1}y^{-1}xy=x^{-1}x^y$. Commutators of higher weight are defined inductively by $[x_1,\ldots ,x_n]=[[x_1,\ldots ,x_{n-1}],x_n]$
for $x_1,\ldots ,x_n\in G$. 
We use shorthand notation $[x,{}_ny]$ for the commutator $[x,y,\ldots ,y]$ with $n$ copies of $y$.
When referring to the classifiation of groups of order $p^5$, we closely follow the notations of James \cite{Jam80}.


\section{The nonabelian exterior square of a group}
\label{s:exterior}

\noindent
We recall the definition and basic properties of the nonabelian exterior square of a group. The reader is referred to \cite{Bro87,Mil52} for more
thorough accounts on the theory and its generalizations.
Let $G$ be a group. We form the group $G\wedge G$, generated by the symbols
$m\wedge n$, where $m,n\in G$, subject to the following relations:
\begin{align}
\label{eq:tens1}
m_1m\wedge n &= (m_1^m\wedge n^m)(m\wedge n),\\
\label{eq:tens2}
m \wedge n_1n &= (m\wedge n)(m^n\wedge n_1^n),\\
\label{eq:tens3}
m\wedge m &= 1,
\end{align}
for all $m,m_1,n,n_1\in G$.
The group $G\wedge G$ is said to be
the {\it nonabelian exterior square} of $G$.
By definition, the commutator map $\kappa :G\wedge G\to [G,G]$, given by
$g\wedge h\mapsto [g,h]$, is a well defined homomorphism of groups.
Clearly $\M(G)=\ker\kappa$ is central in $G\wedge G$, and
$G$ acts trivially via diagonal action on  $\M(G)$.
Miller \cite{Mil52} proved that there is a natural isomorphism between $\M(G)$ and $\HH _2(G,\mathbb{Z})$.

An alternative way of obtaining $G\wedge G$ is the following.
Let $G$ be a group and let $G^\varphi$
be an isomorphic copy of $G$ via the mapping
$\varphi : g \mapsto g^\varphi$ for all $g \in G$.
We define the group $\tau (G)$ to be
$$
\tau(G)=\langle G,G^\varphi \mid
 [g,h^\varphi] ^x =
[g^x,(h^x)^\varphi] = [g,h^\varphi]^{x^\varphi},
[g,g^\varphi]=1\,
\forall x,g,h \in G\rangle .
$$
The groups $G$ and $G^\varphi$ embed into $\tau (G)$. 
In the rest of the paper we consider the group $[G,G^\varphi ]$ as a subgroup
of $\tau (G)$.
Then the map
$\phi:G\wedge G\to [G,G^\varphi]$
defined by $(g\wedge h)^\phi = [g,h^{\varphi}]$ for all $g$ and $h$ in $G$
is an isomorphism.
This was proved independently by 
Ellis and Leonard \cite{Ell95} and Rocco \cite{Roc91}.

The advantage of the above description of $G\wedge G$ is the ability of using the
full power of the commutator calculus instead of computing with elements of $G\wedge G$.
The following lemma collects various properties of $\tau (G)$ and $[G,G^\varphi ]$ that will be used in the proof
of the main result.

\begin{lemma}[\cite{Bly09}]
\label{l:tau}
Let $G$ be a group.
\begin{enumerate}
\item[(a)] If $G$ is nilpotent of class $c$, then $\tau (G)$ is nilpotent of class at most $c+1$.
\item[(b)] If $[G,G]$ is nilpotent of class $c$, then $[G,G^\varphi ]$ is nilpotent of class $c$ or $c+1$.
\item[(c)] If $G$ is nilpotent of class $\le 2$, then $[G,G^\varphi ]$ is abelian.
\item[(d)] $[g,h^\varphi ]=[g^\varphi ,h]$ for all $g,h\in G$.
\item[(e)] $[g,h,k^\varphi ]= [g,h^\varphi ,k]=[g^\varphi,h,k]=[g^\varphi ,h^\varphi ,k]=[g^\varphi ,h,k^\varphi]=[g,h^\varphi ,k^\varphi]$
for all $g,h,k\in G$.
\item[(f)] $[[g_1,h_1^\varphi],[g_2,h_2^\varphi ]]=[[g_1,h_1],[g_2,h_2]^\varphi ]$ for all $g_1,h_1,g_2,h_2\in G$.
\item[(g)] If $g$ and $h$ are commuting elements of $G$ of orders $m$ and $n$, respectively, then the order of $[g,h^\varphi]$ divides
$\gcd (m,n)$.
\end{enumerate}
\end{lemma}

Let $G$ be a finite solvable group. Then $G$ is {\it polycyclic}, i.e., it has a subnormal series
$G=G_1\triangleright G_2\triangleright\cdots\triangleright G_{n+1}=1$ such that every factor $G_{i}/G_{i+1}$ is cyclic of order $r_i$.
A {\it polycyclic generating sequence} of $G$ is a sequence $g_1, \ldots ,g_n$ of elements of $G$ such that $G_i = \langle G_{i+1}, g_i\rangle$ 
for all $1 \le i \le n$. The value $r_i$ is called the {\it relative order} of $g_i$.
With respect to a given polycyclic generating sequence $g_1,\ldots , g_n$, each element $g$ of $G$ can be represented in a unique way ({\it normal form}) as a product $g = g_1^{e_1} g_2^{e_2}\cdots g_n^{e_n}$ with exponents $e_i \in \{0,\ldots , r_i-1\}$.
Given a polycyclic generating sequence, the group $G$ can be presented by a {\it polycyclic presentation}, cf. \cite{Sim94} for further details. 

The following lemma will be basic in our considerations. 

\begin{lemma}[Proposition 20 of \cite{Bly09}]
\label{l:genset}
Let $G$ be a finite solvable group with a polycyclic generating sequence $g_1,\ldots ,g_n$. Then the group $[G,G^\varphi ]$, considered as a subgroup
of $\tau (G)$, is generated by the set $\{ [g_i,g_j^\varphi ] \mid i,j=1,\ldots ,n, i>j\}$.
\end{lemma}


\section{Groups of order $p^5$}
\label{s:p5}

\noindent
Groups of order $p^5$ were classified by James \cite{Jam80}. He compiled a list of polycyclic presentations of these groups, and
divided the non-abelian ones into families denoted by $\Phi _2,\ldots ,\Phi _{10}$, according to isoclinism. In the case $p=3$, James uses the notation $\Delta$
instead of $\Phi$.

As already mentioned, there are no groups $G$ of order $2^5$ with $\B_0(G)\neq 0$, see \cite{Chu08}.
Also, there are precisely three groups of order $3^5$ with nontrivial Bogomolov multiplier. These groups are described
in \cite{Hos11} and \cite{Mor11}. 
According to James's notation, they are the groups $\Delta _{10}(2111)a_i$, $i=1,2,3$.
From here on we assume that $p\ge 5$.  In this section we prove Theorem \ref{t:p5b0} using the theory developed in \cite{Mor11}. 
The main idea is the following. Suppose that $G$ is a finite group.
Let 
$\phi$ be the natural isomorphism between $G\wedge G$ and
$[G,G^\varphi ]$. Let $\kappa ^*:[G,G^\varphi]\to [G,G]$ be the composite of $\phi^{-1}$ and $\kappa$, and denote $\M^*(G)=\M (G)^\phi=\ker\kappa^*$ and $\M ^*_0(G)=\M _0(G)^\phi$. Then $\B _0(G)$ is clearly isomorphic to $\M^*(G) /\M^*_0(G)$
by \cite{Mor11}. 
In order to prove that $\B_0(G)=0$ for a given group $G$ it suffices to show that $\M^*(G) =\M^*_0(G)$. This can be achieved by finding a 
generating set of $\M^*(G)$ consisting solely of elements of $\M^*_0(G)$.

\begin{lemma}
\label{l:class3}
Let $G$ be a nilpotent group of class $\le 3$.
Then $$[x,y^n]=[x,y]^n[x,y,y]^{{n\choose 2}}[x,y,y,y]^{{n\choose 3}}$$ for all $x,y\in \tau (G)$ and every positive integer $n$.
\end{lemma}

\begin{proof}
Since the class of $G$ is at most 3, it follows that $\tau (G)$ is nilpotent of class $\le 4$ by Lemma \ref{l:tau} (a). In particular if $x,y\in \tau (G)$, then $\langle x,y\rangle$ is metabelian. Thus (a) follows easily by induction on $n$.
\end{proof}

In the following, we say that an element $g$ of a polycyclic generating sequence of $G$ is {\it absolute} if its relative order is equal to the order of $g$.

\begin{proposition}
\label{p:class3}
Let $G$ be a $p$-group, $p>3$, and suppose that $G$ is nilpotent of class $\le 3$. Let $g_1,\ldots ,g_n$ be a 
polycyclic generating sequence of $G$. 
Suppose that all nontrivial commutators $[g_i,g_j]$, where $i>j$, are different absolute elements of the polycyclic generating sequence.
Then $\B _0(G)=0$.
\end{proposition}

\begin{proof}
The group $[G,G^\varphi ]$ is generated by the set $\mathcal{G}=\{ [g_i,g_j^\varphi ]\mid i>j\}$. 
As above, let $\kappa ^*$ be the commutator map $[G,G^\varphi ]\to [G,G]$, and suppose that
$[g_{i_1},g_{j_1}],\ldots ,[g_{i_\ell},g_{j_\ell}]$ are the nontrivial elements of $\mathcal{G}^{\kappa ^*}$.
Since $[G,G^\varphi ]$ is nilpotent of class $\le 2$, every word $w\in [G,G^\varphi ]$ can be written as
$$w=\prod _{k=1}^\ell [g_{i_k},g_{j_k}^\varphi ]^{n_k} \cdot \tilde{w},$$
where $\tilde{w}\in \M ^*_0(G)$. Mapping $w$ with $\kappa ^*$, we get
$$w^{\kappa ^*} =\prod _{k=1}^\ell[g_{i_k},g_{j_k}] ^{n_k}.$$
Let $p^{r_k}$ be the  order of $[g_{i_k},g_{j_k}]$. Since these commutators are different absolute terms of the polycyclic generating sequence of $G$,
it follows that $w\in \M^*(G)=\ker\kappa ^*$ if and only if $p^{r_k}$ divides $n_k$ for every $i$. We have that $[g_{i_k},g_{j_k}^\varphi ]^{p^{r_k}}$ belongs
to $\M ^*(G)$. On the other hand,
\begin{align*}
[g_{i_k},g_{j_k}^\varphi ]^{p^{r_k}} &= [g_{i_k}^{p^{r_k}},g_{j_k}^\varphi ][g_{j_k}^\varphi ,g_{i_k},g_{i_k}]^{{p^{r_k} \choose 2}}[g_{j_k}^\varphi ,g_{i_k},g_{i_k},g_{i_k}]^{{p^{r_k} \choose 3}}\\
&= [g_{i_k}^{p^{r_k}},g_{j_k}^\varphi ][[g_{j_k}^\varphi ,g_{i_k}]^{{p^{r_k} \choose 2}},g_{i_k}][[g_{j_k}^\varphi ,g_{i_k}]^{{p^{r_k} \choose 3}},g_{i_k},g_{i_k}].
\end{align*}
Since $p>3$, it follows that 
$[[g_{j_k}^\varphi ,g_{i_k}]^{{p^{r_k} \choose 2}},g_{i_k}]$ and $[[g_{j_k}^\varphi ,g_{i_k}]^{{p^{r_k} \choose 3}},g_{i_k},g_{i_k}]$ belong to
$\M ^*_0(G)$, hence $[g_{i_k}^{p^{r_k}},g_{j_k}^\varphi ]\in\M ^*_0(G)$. Thus we conclude that
$[g_{i_k},g_{j_k}^\varphi ]^{p^{r_k}}$ belongs
to $\M ^*_0(G)$. By the above we clearly have that
$\M^*(G)=\langle [g_{i_k},g_{j_k}^\varphi ]^{p^{r_k}} \mid k=1,\ldots ,\ell\rangle \M^*_0(G)= \M^*_0(G)$ and therefore $\B_0(G)=0$.
\end{proof}

Applying Proposition \ref{p:class3} and inspecting the polycyclic presentations of groups of order $p^5$ in \cite{Jam80}, we obtain the following.

\begin{corollary}
\label{c:phi2346}
Let $G$ be a group of order $p^5$, $p>3$. If $G$ belongs to one of the families $\Phi _2$, $\Phi _3$, $\Phi _4$ or $\Phi _6$, then $\B _0(G)=0$.
\end{corollary}

We deal with the remaining families separately. At first we exhibit those with trivial Bogomolov multiplier.

\begin{proposition}
\label{p:phi5}
Let $G$ be a group of order $p^5$, $p>3$. If $G$ belongs to the family $\Phi _5$, then $\B_0(G)=0$.
\end{proposition}

\begin{proof}
There are two groups in the family $\Phi _5$, by James's notations \cite{Jam80}
we denote them by $\Phi _5(2111)$ and $\Phi _5(1^5)$. They are both nilpotent of class
$2$. We only prove the result for $G=\Phi _5(2111)$, the proof for $\Phi _5(1^5)$ is almost identical. The group $G$ has a polycyclic presentation
$$G=\langle \alpha _1,\alpha _2,\alpha _3,\alpha _4,\beta \mid [\alpha _1,\alpha _2]=[\alpha _3,\alpha _4]=\alpha _1^p=\beta, \alpha _2^p=\alpha _3^p=
\alpha _4^p=\beta ^p=1\rangle,$$
where all the relations of the form $[x,y]=1$ between the generators have been omitted from the list. A generating set of the group $[G,G^\varphi ]$ 
can be found by Lemma \ref{l:genset}. It consists of 
$[\alpha _1,\alpha _2^\varphi ]$, $[\alpha _3 ,\alpha _4^\varphi ]$, $[\alpha _1,\alpha _3^\varphi ]$, $[\alpha _1,\alpha _4^\varphi ]$,
$[\alpha _2,\alpha _3^\varphi ]$, $[\alpha _2,\alpha _4^\varphi ]$, and $[\alpha _i,\beta^\varphi ]$, $i=1,2,3,4$. Apart from the first two, all of these
generators belong to $\M ^*_0(G)$. Since $[G,G^\varphi ]$ is abelian, every element $w\in [G,G^\varphi ]$ can be written as 
$w=[\alpha _1,\alpha _2^\varphi ]^m[\alpha _3 ,\alpha _4^\varphi ]^n\tilde{w}$, where $\tilde{w}\in \M ^*_0(G)$. Then
$w^{\kappa ^*}=\beta ^{m+n}$, hence $w\in \M^*(G)$ if and only if $p$ divides $m+n$. Now, since $\tau (G)$ is nilpotent of class $\le 3$, we get
$$1=[\alpha _1^\varphi ,\alpha _2^p]=[\alpha _1^\varphi, \alpha _2]^p[\alpha _1^\varphi, \alpha _2,\alpha _2]^{{p \choose 2}}=[\alpha _1^\varphi, \alpha _2]^p,$$
hence $[\alpha _1, \alpha _2^\varphi]$ is of order $p$. Similarly, $[\alpha _3, \alpha _4^\varphi]$ is of order $p$. It follows that
$$\M ^*(G)=\langle [\alpha _1,\alpha _2^\varphi ][\alpha _3,\alpha _4^\varphi ]^{-1}\rangle \M^*_0(G).$$
Note that $1=[\alpha _1,\alpha _2][\alpha _4,\alpha _3]=[\alpha _1\alpha _4,\alpha _2\alpha _3]$, hence
$[\alpha _1\alpha _4,(\alpha _2\alpha _3)^\varphi]\in \M ^*_0(G)$. Expanding the latter using the class restriction and Lemma \ref{l:tau}, we get
\begin{align*}
[\alpha _1\alpha _4,(\alpha _2\alpha _3)^\varphi] &=
[\alpha _1,\alpha _2^\varphi\alpha _3^\varphi ][\alpha _1,\alpha _2^\varphi\alpha _3^\varphi ,\alpha _4][\alpha _4,\alpha _2^\varphi\alpha _3^\varphi ]\\
&= [\alpha _1,\alpha _2^\varphi\alpha _3^\varphi][\alpha _1,\alpha _2\alpha _3,\alpha _4^\varphi][\alpha _4,\alpha _3^\varphi][\alpha _4,\alpha _2^\varphi]
[\alpha _4,\alpha _2^\varphi, \alpha _3^\varphi]\\
&= [\alpha _1,\alpha _3^\varphi ][\alpha _1,\alpha _2^\varphi ][\beta ,\alpha _3^\varphi][\beta ,\alpha _4^\varphi ][\alpha _4,\alpha _3^\varphi][\alpha _4,\alpha _2^\varphi ].
\end{align*}
Note that $[\alpha _1,\alpha _3^\varphi]$, $[\beta ,\alpha_3^\varphi]$, $[\beta ,\alpha _4^\varphi ]$, and $[\alpha _4,\alpha _2^\varphi]$
belong to $\M ^*_0(G)$.
From here it follows that $[\alpha _1,\alpha _2^\varphi ][\alpha _3,\alpha _4^\varphi ]^{-1}\in \M^*_0(G)$, hence $\B _0(G)=0$.
\end{proof}

\begin{proposition}
\label{p:phi7}
Let $G$ be a group of order $p^5$, $p>3$. If $G$ belongs to the family $\Phi _7$, then $\B_0(G)=0$.
\end{proposition}

\begin{proof}
The family $\Phi _7$ consists of five groups. James \cite{Jam80} denotes them by
$\Phi _7(2111)a$, $\Phi _7(2111)b_r$, where $r=1$ or $r$ is the smallest non-quadratic residue mod $p$, $\Phi _7(2111)c$, and $\Phi _7(1^5)$.
These groups are all nilpotent of class $3$.

At first we deal with the case $G=\Phi _7(1^5)$.
The group $G$ has a polycyclic presentation
\begin{multline*}
G=\langle \alpha ,\alpha _1,\alpha _2,\alpha _3,\beta \mid [\alpha _i,\alpha ]=\alpha _{i+1}, [\alpha _1,\beta ]=\alpha _3,\\
\alpha ^p=\alpha _1^{(p)}=\alpha _{i+1}^p=\beta ^p=1, (i=1,2)\rangle ,
\end{multline*}
where $\alpha _1^{(p)}$ denotes $\alpha _1^p\alpha _2^{{p\choose 2}}\alpha _3^{{p\choose 3}}$. Since $p>3$, the relation $\alpha _1^{(p)}=1$, together
with other power relations, implies $\alpha _1^p=1$.

The group $[G,G^\varphi ]$ is generated modulo $\M ^*_0(G)$ by $[\alpha _1,\alpha ^\varphi]$, $[\alpha _2,\alpha ^\varphi]$, and 
$[\alpha _1,\beta ^\varphi]$. Since the nilpotency class of $[G,G^\varphi ]$ is at most 2, every element $w\in [G,G^\varphi ]$ can be written as
$w=[\alpha _1,\alpha ^\varphi]^{m_1}[\alpha _2,\alpha ^\varphi]^{m_2}[\alpha _1,\beta ^\varphi]^{m_3}\tilde{w}$, where $\tilde{w}\in\M^*_0(G)$.
This gives $w^{\kappa ^*}=\alpha _2^{m_1}\alpha _3^{m_2+m_3}$, therefore $w\in\M ^*(G)$ if and only if $p$ divides $m_1$ and $m_2+m_3$. By Lemma \ref{l:class3} we have for $i=1,2$ that
\begin{align*}
1 &= [\alpha^\varphi ,\alpha _i^p]\\
&= [\alpha ^\varphi ,\alpha _i]^p[\alpha ^\varphi ,\alpha _i,\alpha _i]^{{p\choose 2}}[\alpha ^\varphi ,\alpha _i,\alpha _i,\alpha _i]^{{p\choose 3}}\\
&= [\alpha ^\varphi,\alpha _i]^p[\alpha _{i+1},\alpha _i^\varphi]^{{p \choose 2}}\\
&= [\alpha ^\varphi,\alpha _i]^p,
\end{align*}
and similarly $[\alpha _i,\beta ^\varphi ]^p=1$. From here it follows that
$\M ^*(G)$ is generated modulo $\M ^*_0(G)$ by $[\alpha _2,\alpha ^\varphi][\alpha _1,\beta ^\varphi]^{-1}$, thus we need to show that
$[\alpha _2,\alpha ^\varphi][\alpha _1,\beta ^\varphi]^{-1}\in \M^*_0(G)$. At first we observe that
$[\alpha _2\beta,\alpha\alpha _1]=[\alpha _2,\alpha _1]^\beta[\alpha _2,\alpha]^{\alpha _1\beta}[\beta ,\alpha _1][\beta ,\alpha ]^{\alpha _1}=1$,
hence $[\alpha _2\beta,(\alpha\alpha _1)^\varphi]\in\M ^*_0(G)$. Expansion gives
\begin{align*}
[\alpha _2\beta,(\alpha\alpha _1)^\varphi] &=
[\alpha _2,\alpha _1^\varphi]^\beta[\alpha _2,\alpha^\varphi]^{\alpha _1^\varphi\beta}[\beta ,\alpha _1^\varphi][\beta ,\alpha^\varphi ]^{\alpha _1^\varphi}\\
&= [\alpha _2,\alpha _1^\varphi]^\beta[\alpha _2,\alpha^\varphi][\alpha _2,\alpha^\varphi,\alpha _1^\varphi\beta][\beta ,\alpha _1^\varphi][\beta ,\alpha^\varphi ]^{\alpha _1^\varphi}.
\end{align*}
Observe that $[\alpha _2,\alpha _1^\varphi ]$, $[\alpha _2,\alpha ^\varphi, \alpha _1^\varphi\beta ]$, and $[\beta ,\alpha ^\varphi]$ all belong
to $\M ^*_0(G)$.
Thus we conclude that $[\alpha _2,\alpha ^\varphi][\alpha _1,\beta ^\varphi]^{-1}=[\alpha _2,\alpha^\varphi][\beta ,\alpha _1^\varphi]\in \M^*_0(G)$, as required.

The proofs for $\Phi _7(2111)a$ and $\Phi _7(2111)c$ are almost the same and are thus omitted. In the case of $G=\Phi _7(2111)b_r$, a polycyclic
presentation of $G$ is
\begin{multline*}
G=\langle \beta _1,\ldots ,\beta _5\mid [\beta _2,\beta _1]=\beta _3, [\beta _3,\beta _1]=\beta _4, [\beta _2,\beta _5]=[\beta _3,\beta _5]=\beta _4,\\
\beta _2^{(p)}=\beta _4^r, \beta _1^p=\beta _3^p=\beta _4^p=\beta _5^p=1\rangle.
\end{multline*}
It is now easy to adapt the above argument to show that $\B _0(G)=0$, hence we skip the details.
\end{proof}

\begin{proposition}
\label{p:phi8}
Let $G$ be a group of order $p^5$, $p>3$. If $G$ belongs to the family $\Phi _8$, then $\B_0(G)=0$.
\end{proposition}

\begin{proof}
The only group in the family $\Phi _8$ is
$$G=\Phi _8(32)=\langle \alpha _1,\alpha _2,\beta \mid [\alpha _1,\alpha _2]=\beta =\alpha _1^p, \beta ^{p^2}=\alpha _2^{p^2}=1\rangle .$$
$G$ is nilpotent of class $3$, and $[G,G^\varphi ]$ is nilpotent of class $\le 2$, generated modulo $\M^*_0(G)$ by $[\alpha _1,\alpha _2^\varphi]$.
We have
\begin{align*}
1 &= [\alpha _1^\varphi,\alpha _2^{p^2}]\\
&=[\alpha _1^\varphi,\alpha _2]^{p^2}[\alpha _1^\varphi,\alpha _2,\alpha _2]^{p^2\choose 2}[\alpha _1^\varphi,\alpha _2,\alpha _2,\alpha _2]^{p^2\choose 3}\\
&= [\alpha _1^\varphi,\alpha _2]^{p^2}[\beta ,\alpha _2^\varphi]^{p^2\choose 2}[\beta ,\alpha _2,\alpha _2^\varphi]^{p^2\choose 3}\\
&=[\alpha _1^\varphi,\alpha _2]^{p^2}.
\end{align*}
Every element $w\in [G,G^\varphi]$ can be written as $w=[\alpha _1,\alpha _2^\varphi]^m\tilde{w}$ with $\tilde{w}\in\M^*_0(G)$. Thus $w$ belongs to
$\M^*(G)$ if and only if $\beta ^m=1$, which in turn is equivalent to $p^2|m$. It follows that $\M^*(G)=\M^*_0(G)$, hence $\B_0(G)=0$.
\end{proof}

The families $\Phi _9$ and $\Phi _{10}$ consist of groups $G$ of order $p^5$ that are nilpotent of class $4$. Since every group of order $p^5$ is
metabelian, it follows that $[G,G^\varphi ]$ is nilpotent of class $\le 2$. Besides, $\tau (G)$ is nilpotent of class $\le 5$. When expanding
commutators in $\tau (G)$, we will need the following lemma.

\begin{lemma}[Lemma 9 of \cite{Mor08}]
\label{l:class5}
Let $H$ be a nilpotent group of class $\le 5$, $n$ a positive integer and $x,y\in H$. Then
$$[x^n,y]=[x,y]^n[x,y,x]^{{n\choose 2}}[x,y,x,x]^{{n\choose 3}}[x,y,x,x,x]^{{n\choose 4}}[x,y,x,[x,y]]^{\sigma (n)},$$
where $\sigma (n)=n(n-1)(2n-1)/6$.
\end{lemma}

\begin{proposition}
\label{p:phi9}
Let $G$ be a group of order $p^5$, $p>3$. If $G$ belongs to the family $\Phi _9$, then $\B_0(G)=0$.
\end{proposition}

\begin{proof}
The groups belonging to the family $\Phi _9$ are denoted by $\Phi _9(2111)a$, $\Phi _9(2111)b_r$, where $r+1=1,2,\ldots ,\gcd (p-1,3)$, and $\Phi _9(1^5)$.

All the groups in $\Phi _9$ have a polycyclic generating sequence $\alpha ,\alpha _1,\ldots ,\alpha _4$ satisfing the relations 
$[\alpha _i,\alpha ]=\alpha _{i+1}$, where $i=1,2,3$. These are the only nontrivial commutator relations between these generators. Thus
the group $[G,G^\varphi ]$ is generated modulo $\M^*_0(G)$ by
$[\alpha _i,\alpha ^\varphi ]$, where $i=1,2,3$. Since the nilpotency class of $[G,G^\varphi ]$ is at most 2, every element $w$ of $[G,G^\varphi ]$ can be
written as 
$w=\prod _{i=1}^3 [\alpha _i,\alpha ^\varphi ]^{m_i}\tilde{w},$
where $\tilde{w}\in\M ^*_0(G)$. We have that $w^{\kappa ^*}=\prod _{i=1}^3\alpha _{i+1}^{m_i}$. In all cases, the elements $\alpha _{i+1}$, where $i=1,2,3$, have order
$p$, hence $w\in \M^*(G)$ if and only if $p|m_i$, $i=1,2,3$.

In the cases when $G=\Phi _9(1^5)$ or $G=\Phi _9(2111)a$, the power relations imply that $\alpha _1^p=1$. For $i=1,2,3$ we get by Lemma \ref{l:class5} that
\begin{align*}
1 &= [\alpha _i^p,\alpha ^\varphi ]\\
&= [\alpha _i,\alpha^\varphi ]^p[\alpha _i,\alpha^\varphi,\alpha _i ]^{p\choose 2}[\alpha _i,\alpha^\varphi,{}_2\alpha _i ]^{p\choose 3}
[\alpha _i,\alpha^\varphi,{}_3\alpha _i ]^{p\choose 4}[\alpha _i,\alpha^\varphi,\alpha _i,[\alpha _i,\alpha^\varphi ]]^{\sigma (p)}\\
&=[\alpha _i,\alpha^\varphi ]^p[\alpha _{i+1},\alpha _i^\varphi ]^{p\choose 2}[\alpha _{i+1},\alpha _i,\alpha _i^\varphi ]^{p\choose 3}
[\alpha _{i+1},\alpha _i,\alpha _i,\alpha _i^\varphi ]^{p\choose 4}[\alpha _{i+1},\alpha _i^\varphi,[\alpha _i,\alpha^\varphi ]]^{\sigma (p)}\\
&= [\alpha _i,\alpha^\varphi ]^p.
\end{align*}
From here it follows that $\M^*(G)=\M^*_0(G)$, hence $\B _0(G)=0$.

When $G=\Phi _9(2111)b_r$, the defining relations imply that $\alpha ^p=\alpha _1^{p^2}=1$. By the above calculations we see that $[\alpha _{i+1},\alpha^\varphi]$, $i=1,2,3$, are elements of order $p$. On the other hand, expansion of $1=[\alpha ^p,\alpha _1^\varphi ]$ yields that $[\alpha _1,\alpha^\varphi ]$ has order $p$ as well. Therefore $\M^*(G)=\M^*_0(G)$, and thus $\B _0(G)=0$.
\end{proof}

By Theorem \ref{t:hoshi}, this finishes the proof of Theorem \ref{t:p5b0}.


\end{document}